\documentclass{amsart}
\usepackage[T1]{fontenc}
\usepackage[french]{babel}
\usepackage{amssymb, amsmath}
\usepackage[hypertex]{hyperref}
\usepackage[all]{xy}

\numberwithin{equation}{section}

\theoremstyle{plain}
\newtheorem{theorem}{Th\'eorème}[section]
\newtheorem{cor}[theorem]{Corollaire}
\newtheorem{prop}[theorem]{Proposition}
\newtheorem{lemma}[theorem]{Lemme}

\theoremstyle{definition}

\newtheorem{remark}[theorem]{Remarque}
\newtheorem{fact}[theorem]{Fait}
\newtheorem{definition}[theorem]{D\'efinition}

\newcommand{\acl}{\operatorname{acl}}
\newcommand{\aclq}{\operatorname{acl}^\mathrm{eq}}

\newcommand{\inv}[1]{ #1^{-1}}

\newcommand{\tp}{\operatorname{tp}}

\def\tp{\mathrm{tp}}

\def\prod{\! \cdot \!}
\def\a{\overline{a}}
\def\ab{\overline{a \prod b}}
\def\subab{\overline{a \cdot b}}
\def\ca{\overline{c \prod a }}
\def\cab{\overline{ c \prod a \prod b}}
\def\b{\overline{b}}
\def\cc{\overline{c}}
\def\d{\overline{d}}
\def\e{\overline{e}}
\def\f{\overline{f}}
\def\coker{\mathrm{coker}}

\def\Ind#1#2{#1\setbox0=\hbox{$#1x$}\kern\wd0\hbox to 0pt{\hss$#1\mid$\hss}
\lower.9\ht0\hbox to 0pt{\hss$#1\smile$\hss}\kern\wd0}
\def\Notind#1#2{#1\setbox0=\hbox{$#1x$}\kern\wd0\hbox to
0pt{\mathchardef\nn="0236\hss$#1\nn$\kern1.4\wd0\hss}\hbox to
0pt{\hss$#1\mid$\hss}\lower.9\ht0
\hbox to 0pt{\hss$#1\smile$\hss}\kern\wd0}
\def\ind{\mathop{\mathpalette\Ind{}}}
\def\nind{\mathop{\mathpalette\Notind{}}}
\def\indip{\mathop{\ \ \hbox to 0pt{\hss$\mid^{\hbox to
0pt{$\scriptstyle{P}$\hss}}$\hss}
\lower4pt\hbox to 0pt{\hss$\smile$\hss}\ \ }}
\def\nindip{\mathop{\ \ \hbox to 0pt{\hss$\!\not{\mid}^{\hbox to
0pt{$\scriptstyle\,{P}$\hss}}$\hss}
\lower4pt\hbox to 0pt{\hss$\smile$\hss}\ \ }}

\def\indild{\mathop{\ \ \hbox to 0pt{\hss$\mid^{\hbox to
0pt{$\scriptstyle\mathrm{ld}$\hss}}$\hss}
\lower4pt\hbox to 0pt{\hss$\smile$\hss}\ \ }}

\begin{document}

\title{De beaux groupes}
\date{\today}

\author{Thomas Blossier et Amador Martin-Pizarro}
\address{Universit\'e de Lyon; CNRS; Universit\'e Lyon 1; Institut Camille Jordan UMR5208, 43 boulevard du 11
novembre 1918, F--69622 Villeurbanne Cedex, France }
\email{blossier@math.univ-lyon1.fr}
\email{pizarro@math.univ-lyon1.fr}
\thanks{Recherche conduite dans le cadre du projet MODIG ANR-09-BLAN-0047}
\keywords{Model Theory, Groups, Pairs}
\subjclass{03C45}

\begin{abstract} Dans une belle paire $(M,E)$ de modèles d'une
théorie stable $T$ ayant élimination des imaginaires sans la propriété
de recouvrement fini, tout groupe définissable  se
projette, à isogénie près, sur les points $E$-rationnels d'un groupe
définissable dans le réduit à paramètres dans $E$. Le  noyau de cette
projection est un groupe définissable dans le réduit. 

Un groupe interprétable dans une paire $(K,F)$ de corps
algébriquement clos où $K$ est une extension propre de $F$ est, à
isogénie près,
l'extension des points $F$-rationnels d'un groupe algébrique sur $F$
par un groupe interprétable quotient d'un groupe
algébrique par les points $F$-rationnels d'un sous-groupe algébrique,
le tout défini sur $F$.
\end{abstract}
\maketitle

\section*{English Summary} In this short paper, we characterise
definable groups in a \emph{belle paire} $(M,E)$ of models of a
stable theory $T$ having elimination of imaginaries without the finite
cover property: every definable group is (up to isogeny) the extension
of the $E$-rational points of a group definable in the theory $T$ over
$E$ by a group definable in $T$.  

Furthermore, if $F\subsetneq K$ is a proper extension of algebraically
closed fields, every interpretable group in the pair $(K,F)$ is, up to
isogeny, the
extension of the subgroup of $F$-rational points of an algebraic
group over $F$ by an interpretable group which is the quotient of an
algebraic group by the $F$-rational points of an algebraic subgroup.

\section*{Introduction}

Deux corps algébriquement clos de même caractéristique $p$ satisfont
les mêmes propriétés élémentaires du premier ordre (cette collection
de propriétés forme alors une \emph{théorie complète}, notée
$ACF_p$). Un argument de compacité classique permet, par exemple, de
montrer que toute fonction polynomiale injective à
plusieurs variables complexes est surjective~\cite{Ax69}. La théorie
des corps algébriquement clos de caractéristique fixée, l'un
des archétypes des théories \emph{stables}, a un comportement fortement
structuré : la clôture algébrique d'un sous-corps $k$
correspond à la \emph{clôture algébrique} modèle-théorique (l'ensemble
des points à orbite fini par les automorphismes fixant $k$)~;
la dépendance algébrique correspond à la \emph{déviation}
Shelahiste~; le fait, dû à Chevalley, que la
projection d'un ensemble constructible de Zariski le reste,
correspond à l'élimination de quantificateurs dans le
langage des anneaux. De plus, l'existence d'un corps de définition
pour chaque variété équivaut à \emph{l'élimination des imaginaires},
ce qui permet de travailler avec des quotients, objets
\emph{interprétables} dans la théorie.

De nombreuses questions en géométrie algébrique se traduisent de façon
naturelle en théorie des modèles et vice-versa. Cela a instauré une
véritable interaction entre la théorie des modèles et la géométrie 
algébrique. 

On est parfois amené à considérer des objets formels à l'intérieur de
variétés algébriques. Ces objets sont fréquemment décrits à l'aide de
l'algèbre différentielle.  Dans le cas différentiel, les corps
\emph{différentiellement clos}, qui en caractéristique nulle ont
tous la même théorie, sont l'analogue des purs corps
algébriquement clos.
De même, la projection d'un ensemble différentiellement constructible
le reste. En revanche, un corps différentiellement clos possède un
sous-corps propre algébriquement clos, constitué des éléments à
dérivée nulle. Les paires d'extensions propres de corps algébriquement
clos ont été considérées pour elles-mêmes dans les travaux de 
Keisler~\cite{hjK64}, motivés par ceux de Robinson. La théorie
des paires d'extensions propres de corps algébriquement clos de
caractéristique fixée est complète~: elle correspond à la théorie des
\emph{belles paires} de modèles de $ACF_p$.
Plus généralement, les belles paires de structures stables ont été
étudiées par Poizat~\cite{Po83}. Il y axiomatise leurs théories dans le cas où
les structures n'ont pas la \emph{propriété de recouvrement fini}.

Plus récemment, Delon \cite{fD12} a obtenu une expansion naturelle du
langage des anneaux pour l'élimination des quantificateurs des belles
paires de corps algébriquement clos. Sa preuve consiste à imposer aux
sous-structures d'une belle paire $(M,P)$ d'être $P$-indépendantes~: en effet,
pour un uple $P$-indépendant $\bar a$, les formules sans
quantificateur satisfaites par $\bar a$ (qui forment le \emph{type
sans quantificateur}) entraînent toutes les propriétés élémentaires
satisfaites par $\bar a$, son type parital \cite{BPV}.

Bien que les belles paires de corps algébriquement clos soient
$\omega$-stables de rang de Morley $\omega$ \cite[Example 1.5]{Bue91},
et que la déviation y soit bien comprise \cite{Po83,BPV}, elles
n'éliminent pas les imaginaires, même \emph{géométriquement}. En
fait, une belle paire d'une théorie $T$ élimine les imaginaires
(modulo les imaginaires de $T$) si et seulement si aucun groupe infini
n'est interprétable dans $T$~\cite{PiVa04}. 

Pillay \cite{Pi07}
explicite une expansion du langage des
belles paires de corps
algébriquement clos permettant d'éliminer géométriquement les
imaginaires~: tout imaginaire est alors
interalgébrique avec un uple
d'éléments dans les (nouvelles) sortes de l'expansion.

L'étude des groupes interprétables est une question récurrente en
théorie des modèles. Pour traiter cette question, la construction de
groupes à partir d'un diagramme de configuration de groupes introduite
par Hrushovski \cite{eHPhD} s'avère fondamentale. Dans le cas des
corps différentiels, tout groupe interprétable se plonge dans un
groupe algébrique \cite{Pi97}, ce qui est en partie conséquence du fait que  la
clôture algébrique d'un uple au sens différentiel correspond à la
clôture algébrique corpique du corps différentiel engendré. Dans
une belle paire, la clôture algébrique au sens de $T$ d'une structure
$P$-indépendante est algébriquement close au sens parital \cite[Lemma
2.5]{PiVa04}. Cette propriété nous sera fort utile pour décrire les
groupes
définissables dans toute belle paire de structures
stables ainsi que les groupes interprétables dans une belle paire de
corps algébriquement clos (cf. corollaire~\ref{C:intACF} et
proposition~\ref{P:Gdef}). On se servira également des
techniques utilisées dans \cite{BPW13} pour la
classification des groupes définissables  dans les mauvais corps
\cite{BHPW06}.

\section{Préliminaires}

Les résultats cités dans ces préliminaires se trouvent dans
\cite{Po83,BPV,PiVa04}. 
Nous fixons pour toute la suite une théorie complète $T$ stable, que
l'on suppose avec élimination des quantificateurs et des imaginaires,
afin de simplifier la rédaction.  

Une \emph{belle paire} de $T$ est
la donnée de $E \preceq M$, modèles de $T$ tel que $E$ est
$|T|^+$-saturé et $M$ réalise tout type sur $A\cup E$,
où $A \subset M$ est de cardinal strictement inférieur à $|T|^+$.
Puisque deux belles paires sont
toujours élémentairement équivalentes dans le langage de $T$ muni d'un
nouveau prédicat $P$ pour le sous-modèle, on notera $T_P$ leur
théorie. L'indice $P$ fera systématiquement référence à $T_P$.
Par exemple, si $(M,E)$ est une belle paire, on notera $\acl(A)$
les éléments réels de $M$ algébriques sur $A$ au sens de $T$ et
$\acl_P(A)$ les éléments
réels de $M$ algébriques sur $A$ au sens de la paire.
De même, on utilisera les symboles $\ind$ et $\indip$ pour
les indépendances au sens de $T$ et de $T_P$, respectivement.

Une partie $A$ de $M$ est dite $P$\emph{-indépendante} si
$$ A \ind_{E_A} E,$$
où  $E_A = E \cap A =P(A)$. Deux sous-uples $P$-indépendants
satisfaisant le même $T_P$-type sans
quantificateurs  sont élémentairement
équivalents.

Un modèle $|T|^+$-saturé de $T_P$ n'est pas
forcément une belle paire. C'est le cas si et seulement si la théorie
$T$ n'a pas la propriété de recouvrement fini. Cette propriété
correspond à l'existence de relations d'équivalence  uniformément
définissables avec un nombre fini arbitrairement grand de classes.

On suppose à partir de maintenant que la théorie $T$ n'a pas la
propriété de recouvrement fini et on se place dans un modèle $(M,E)$
suffisamment saturé de $T_P$.

En particulier, la description des types pour des parties
$P$-indépendantes ci-dessus entraîne que $T_P$ a même spectre de
stabilité que $T$ et donc 
$T_P$ est stable. Par ailleurs, puisque toute partie de $E$ est
$P$-indépendante, la paire n'induit pas de structure supplémentaire
sur $E$. 

À partir de \cite[Remark 7.2 et Proposition 7.3]{BPV} (\emph{cf.}\ également
\cite[Fact 
2.2 et Lemma 2.6]{PiVa04}), on obtient les
propriétés 
suivantes~:
\begin{fact}\label{F:caractIndependance} Soit $A$ et $B$ deux parties
de $M$.

\begin{itemize}
\item La clôture algébrique $\acl_P(A)$ est toujours $P$-indépendante.
\item Si $A$ est $P$-indépendante, alors $A \indip_{E_A} E$. De
plus, on a $\acl_P(A)= \acl(A)$ et $E_{\acl_P(A)} =
\acl(E_A)=\acl_P(E_A)$.

\item Si $A$ et $B$ sont algébriquement closes au sens de la paire,
alors $$\begin{array}{rcl}A\indip\limits_{A\cap B} B & \mbox{si et
seulement si} &
\left\{\begin{array}{l}
A\ind\limits_{A\cap B,E} B\\[4mm]
\hskip7mm\mbox{et}\\[1.5mm]
E_A \ind\limits_{E_{A\cap B}} E_B
\end{array}\right. 
\end{array}$$

\end{itemize}
\end{fact}

La caractérisation de l'indépendance ci-dessus permet d'obtenir
l'analogue à la propriété $(\dag)$ dans \cite{BPW09}.

\begin{lemma}\label{L:Eamalgam}
\'Etant données deux parties $A$ et $B$ algébriquement closes au sens
de la paire et indépendantes au-dessus de leur intersection, alors 

$$\acl_P(A,B) = \acl(A,B) \text{ et } E_{\acl_P(A,B)} =
\acl(E_A,E_B).$$ 
\end{lemma}

\begin{proof}
Par la caractérisation précédente, 
$$A\ind\limits_{A\cap B,E} B,E.$$ Comme $A\ind_{E_A} E$, il
suit par transitivité que $$A\ind\limits_{A\cap B,E_A} B,E\text{ ,
puis }   
A\ind\limits_{E_A,B} E.$$ 
Enfin, puisque $B\ind\limits_{E_B} E$, on a  $B\ind\limits_{E_B, E_A}
E$ et donc $A,B\ind\limits_{E_A,E_B} E$.

\noindent La partie $A \cup B$ est donc $P$-indépendante. Le fait
précédent permet de conclure que $\acl_P(A,B) = \acl(A,B)$ et
$E_{\acl_P(A,B)} =\acl(E_A,E_B)$.
\end{proof}

\section{Groupes définissables dans la belle paire}
Nous disposons de tous les ingrédients pour décrire les groupes
définissables dans une belle paire de $T$. Une adaptation immédiate de
\cite[proposition 1.8]{BPW13} donne un critère suffisant pour
vérifier quand un sous-groupe $T_P$-définissable d'un groupe
$T$-définissable l'est aussi.

\begin{lemma}\label{L:SgrDefInt}
Soit un translaté d'un sous-groupe $H$ connexe $T_P$-type-définissable d'un
groupe
$T$-type-définissable $G$, le tout défini sur un ensemble
algébriquement clos $A$ au sens de la paire. Soit $a$ réalisant le générique
sur $A$ de ce translaté. Si $E_{\acl_P(a,A)} = E_A$, alors  $H$ est
$T$-type-définissable.  
\end{lemma}

\begin{proof}
Notons tout d'abord que l'on peut supposer
que $a$ est le générique de $H$ sur $A$~: en effet,
en prenant une deuxième réalisation $a'$ du générique sur $A$ du translaté de
$H$ telle que 
$a' \indip_A a$, on conclut que $a'\cdot \inv{a}$ (ou $\inv{a} \cdot a'$)
réalise le générique de 
$H$. Comme $a$ et $a'$ ont le même type sur $A$ au sens de la paire, 
alors $E_{\acl_P(a',A)}=E_A$ aussi. Par le lemme~\ref{L:Eamalgam}, on a 
$E_{\acl_P (a'\cdot \inv{a},A)}\subset \acl(E_{\acl_P(a,A)}, 
E_{\acl_P(a',A)} )=E_A$. 

Posons $p_0$ le $T$-réduit de $p$ et notons $H_0$ l'enveloppe
$T$-définissable de $H$, c'est-à-dire,
le plus petit sous-groupe
$T$-type définissable de $G$ contenant $H$. 
Par \cite[proposition 1.8]{BPW13}, l'enveloppe $H_0$ de
$H$ est $T$-connexe et égale au
$T$-stabilisateur de $p_0$, qui est son unique type générique. 

Il suffit maintenant de montrer que $p$ est l'unique générique de
$H_0$ au sens parital (au sens de $T_P$). D'abord, notons que tout
générique parital de
$H_0$ est une complétion de $p_0$, car toute $T$-formule générique
dans $H_0$ au sens de $T_P$ l'est également au sens de $T$ et donc est
contenue dans $p_0$.

Vérifions ensuite que, si $h$ est un générique parital de
$H_0$ sur $A$, alors  $h,A\ind_{A} E$. Si $h\nind_{A} E$,  il
existe une formule $\varphi(x,e) \in \tp(h/A,E)$, à paramètres sur
$A$, qui n'est pas générique dans $H_0$. Par \cite[Lemme
5.5]{Po87}, on peut supposer que pour chaque uple $b$, la formule
$\varphi(x,b)$ n'est pas générique dans $H_0$. Par contre, la
$T_P$-formule $\psi(x) = \exists y P(y) \wedge
\varphi(x,y)$ à paramètres dans $A$ est réalisée par $h$ et est donc
générique dans $H_0$.
Un nombre fini de translatés de la formule $\psi(x)$ doit recouvrir
le groupe $H_0$ et, en prenant une extension non-déviante de $p$, on
peut supposer que  l'uple $a$ est contenu dans $\alpha\cdot
\psi(x)$ pour un certain $\alpha \in H_0$ avec
$a\indip_A \alpha$. Il existe donc $e' \in E$ tel que $a \in \alpha\cdot
\varphi(x,e')$.  Par  la
caractérisation de l'indépendance, on a
$a\ind_{A,E} \alpha$.
De plus, comme $\acl_P(a,A)$ est $P$-indépendant par le fait
\ref{F:caractIndependance} et $E_{\acl_P(a,A)} = E_A$ par hypothèse,
il suit que $a,A \ind_{E_A} E$ et, par transitivité, $a\ind_{A}
E,\alpha$.
Ceci contredit que la formule $\alpha\cdot \varphi(x,e')$ n'est pas
générique dans $H_0$.

Par transitivité, on a $h,A\ind_{E_A} E$, donc $h,A$ est
$P$-indépendant et $E_{\acl_P(h,A)} =
E_{\acl(h,A)} = \acl(E_A) = E_A$.
Ainsi, $h,A$ est  $P$-indépendant (comme $a,A$) et a même
$T_P$-type sans quantificateurs que $a,A$, donc même $T_P$-type $p$.
Le type $p$ est par conséquent l'unique générique de
$H_0$ au sens parital. On conclut alors que $H=H_0$ est
$T$-type-définissable. 
\end{proof}

\begin{remark}\label{R:cond_nec}
Si $G$ est un groupe infini $T$-type-définissable sur $E$, alors le
sous-groupe $H=G(E)$ des points $E$-rationnels de $G$, qui est
clairement définissable dans la paire mais pas dans le réduit,
ne satisfait pas les conditions du lemme.
\end{remark}

Nous allons caractériser les groupes définissables dans la théorie de la 
paire, à isogénie près. Pour deux groupes type-définissables $G$ et $H$, 
une {\em isogénie} de $G$ dans $H$ est un sous-groupe type-définissable 
$S$ de $G\times H$ tel que :
\begin{itemize}
\item les projections $G_S$ sur $G$ et $H_S$ sur $H$ sont chacune d'indice borné ; et
\item le {\em noyau} $\ker(S)=\{g\in G:(g,1)\in S\}$ et le {\em co-noyau} $\coker(S)=\{h\in H:(1,h)\in S\}$ sont finis.\end{itemize}
Une isogénie induit donc un isomorphisme entre $G_S/\ker(S)$ et $H_S/\coker(S)$.

\begin{remark}\label{R:connexe}
La relation d'isogénie est une relation d'équivalence. Tout groupe 
type-définissable est isogène à sa composante connexe. En particulier, 
une isogénie de la composante connexe de $G$ dans $H$ définit
également une isogénie de $G$ dans $H$ et, réciproquement, une isogénie
entre $G$ et $H$ induit une isogénie entre leurs composantes connexes.
\end{remark}

Le résultat suivant, qui combine \cite[lemme
1.2]{BPW13} (une généralisation de \cite{Zi06} au cas non abélien) et 
\cite[lemme
1.5]{BPW13}, sera utilisé à plusieurs reprises.

\begin{lemma}\label{L:Stab_iso}
 Soient $G_1$ et $G_2$ deux groupes type-définissables (ou même
type-interprétables) dans une théorie stable. S'il existe, sur un
ensemble
de paramètres $C=\aclq(C)$, des éléments  $a_1, b_1$  de $G_1$ et $
a_2,b_2,$ de
$G_2$ tels que 
\begin{enumerate}
 \item\label{I:interalg} $a_1$ et $a_2$, $b_1$ et $b_2$
ainsi que $a_1\cdot b_1$ et $a_2\cdot b_2$ sont $C$-interalgébriques,
\item $a_1$, $b_1$ et $a_1\cdot b_1$
sont deux à deux indépendants sur $C$,
\end{enumerate}
alors  l'élément $a_1$ (resp. $a_2$) est générique dans un
unique translaté d'un sous-groupe connexe $H_1$ de $G_1$ (resp.
$H_2$ de $G_2$), le tout type-définissable sur $C$ et il existe une
isogénie entre $H_1$ et $H_2$, donnée par le stabilisateur de
$\tp(a_1,a_2/C)$

Si dans la condition $(\ref{I:interalg})$, on a uniquement 
$a_2$  algébrique sur $C,a_1$, respectivement
$b_2$  algébrique sur $C,b_1$ et $a_2\cdot b_2$ algébrique
sur $C,a_1\cdot b_1$, alors on obtient, dans la conclusion, une
projection type-interprétable de $H_1$ sur un quotient de $H_2$ par un
sous-groupe fini.
\end{lemma}

Dans les preuves qui vont suivre, nous serons amenés à considérer des
groupes type-interprétables d'arité infinie. De
tels groupes seront dits $*$-interprétables, cas particulier de
groupes hyperdéfinissables (voir par exemple \cite[Section 4.3]{Wa00}
pour les groupes hyperdéfinissables).

\begin{remark}\label{R:Stab_iso_etoile_def}
Le lemme ci-dessus se généralise au cas où les groupes $G_1$ et
$G_2$ sont $*$-interprétables. De plus, si $G_1$ est type-définissable
et $G_2$ est $*$-interprétable, alors, par stabilité, on obtient une
isogénie (resp. une projection de même noyau), vers un groupe
connexe $H'_2$ type-interprétable, dont son générique est
interalgébrique avec celui de $H_2$ au-dessus de $C$.
\end{remark}

\begin{prop}\label{P:Gdef}
Soit $T_P$ la théorie des belles paires d'une théorie $T$ stable
avec élimination des imaginaires sans la propriété de recouvrement
fini. Tout groupe $T_P$-type-définissable est isogène à un sous-groupe
d'un groupe $T$-type-définissable. De plus, tout groupe
$T_P$-type-définissable est, à isogénie près, l'extension des points
$E$-rationnels d'un groupe $T$-type-définissable sur $E$ par un groupe
$T$-type-définissable. 
\end{prop}

\begin{proof}
Par la remarque \ref{R:connexe}, il suffit de considérer
un groupe connexe $G$ type-définissable dans une belle paire
$(M,E)$ sur un ensemble de paramètres. Par la suite, on travaillera
au-dessus d'un sous-modèle contenant
ces paramètres, que l'on omettra. \'Etant donnés deux génériques
indépendants $a$ et $b$ de $G$, on dénotera par $\a$, $\b$ et $\ab$
les cl\^otures algébriques au sens parital des points $a$, $b$ et
$a \cdot b$
respectivement. Par
le lemme \ref{L:Eamalgam}, l'uple $\ab$ est $T$-algébrique sur $\a
\cup \b$, car $\a$ et $\b$ sont deux
sous-ensembles algébriquement clos indépendants.

\`A l'aide d'un troisième générique $c$ indépendant de $a$ et $b$,
si
l'on pose $\cc= \acl_P(c)$, $\ca = \acl_P( c \cdot a)$ et $\cab=
\acl_P(c \cdot a \cdot b)$, on 
obtient le diagramme suivant~:

\begin{center}

\begin{picture}(200,110)
\put(0,0){\line(1,1){100}}
\put(200,0){\line(-1,1){100}}
\put(0,0){\line(5,2){143}}
\put(200,0){\line(-5,2){143}}
\put(-17,-2){$\ab$}
\put(203,-2){$\ca$}
\put(97,104){$\a$}
\put(47,55){$\b$}
\put(150,55){$\cc$}
\put(90,47){$\cab$}
\put(-2,-2){$\bullet$}
\put(197,-2){$\bullet$}
\put(97,98){$\bullet$}
\put(98,38){$\bullet$}
\put(54,55){$\bullet$}
\put(141,55){$\bullet$}
\end{picture}
\end{center}

Par \cite[Lemme 2.1]{BPW09}, chaque couple de points sur une droite
ainsi que chaque triplet de points  non colinéaires sont $T$-indépendants. Le
théorème de configuration de groupe \cite{eHPhD} (dont la preuve
s'adapte au cas d'uples infinis)
nous donne un groupe $*$-interprétable connexe au sens de $T$,
dont le générique est $T$-interborné avec $\a$ et donc avec $a$.  Par
le lemme \ref{L:Stab_iso}, sa remarque et l'élimination des
imaginaires de $T$, on peut
supposer, à isogénie près, que le groupe $G$ est un sous-groupe
$T_P$-définissable d'un groupe $T$-type-définissable.

Par le lemme \ref{L:Eamalgam}, la partie $E_{\subab}$ est
$T$-algébrique
sur $E_{\a} \cup E_{\b}$. De plus, la caractérisation de
l'indépendance entraîne que  $E_{\a}$ et $E_{\b}$ sont
$T$-indépendants sur $E_{\a \cap \b}$, qui correspond aux $P$-points
du sous-modèle au-dessus duquel l'on travaille. De façon analogue, on
obtient un groupe $H$ connexe
$*$-interprétable au sens de $T$ défini sur $E$, tel que  $E_{\a}$ est
$T$-interalgébrique avec un $T$-générique $h$ de $H(E)$, qui est
également générique au sens de la paire.

La remarque \ref{R:Stab_iso_etoile_def} donne une projection
$\pi$ de $G$ sur les points $E$-rationnels d'un groupe
type-définissable connexe au sens de $T$, que l'on dénotera
également par $H(E)$,  dont le générique
$h$ est $T_P$-interalgébrique avec $E_{\a}$. 

La composante connexe $N$ du noyau est un groupe
$T$-type-définissable par le lemme \ref{L:SgrDefInt}~: soit $n$ un
générique de $N$ au sens de $T_P$ sur $a$. Alors $n \cdot a$ est un
générique de $N \cdot a$ sur son paramètre canonique, qui est
interalgébrique
avec $E_{\a}$. Comme $(n,1_H)$ est dans le stabilisateur de
$\tp_P(a,h)$, on a $(n \cdot a,h) \equiv^P (a,h)$ et en particulier 
$E_{\a}=\acl_P(h)= E_{\overline{n \cdot a}}$.
L'hypothèse du lemme est ainsi
vérifiée au-dessus de $E_{\a}$. 

\end{proof}

\section{Groupes interprétables dans les beaux corps}

Dans l'introduction, on avait remarqué  qu'une belle paire d'une
théorie stable $T$ n'élimine pas les imaginaires
dès qu'un groupe infini est interprétable dans $T$ \cite{PiVa04}. 
En revanche, la description des groupes définissables nous permet
maintenant de caractériser les groupes interprétables dans une belle
paire (de structures stables) ayant des imaginaires \emph{modérés},
propriété qui est vérifiée par les paires d'extensions propres de corps
algébriquement clos, ou plus généralement, les paires
propres de modèles d'une théorie fortement minimale avec
$\acl(\emptyset)$ infini  \cite{Pi07}.

Pour l'étude de certains groupes type-interprétables dans une
belle paire, nous allons d'abord, à partir d'une
configuration de groupe, trouver un diagramme pour des représentants réels
le plus indépendant possible.

\begin{lemma}\label{L:gpconfig}
Soit $G$ un groupe type-interprétable connexe dans une théorie stable donnée, 
trois génériques indépendants  $\alpha$, $\beta$  et $\gamma$ de
$G$, et un uple (réel) $a_0$. Si $\alpha$ est algébrique sur 
$a_0$, alors il existe des uples (réels) $a$, $b$,
$c$, $d$, $e$ et $f$ tels que 
\begin{itemize}
 \item $ (a,\alpha) \equiv (b, \beta) \equiv (c,\gamma)  \equiv
(d,\alpha\cdot\beta) \equiv (e,\gamma\cdot\alpha) \equiv
(f,\gamma\cdot\alpha\cdot\beta)\equiv (a_0,\alpha)$, et 
 \item $a\ind_{\alpha} b,c,d,e,f$ et de même pour chaque autre point.
\end{itemize}

En particulier,  dans le
diagramme suivant~:

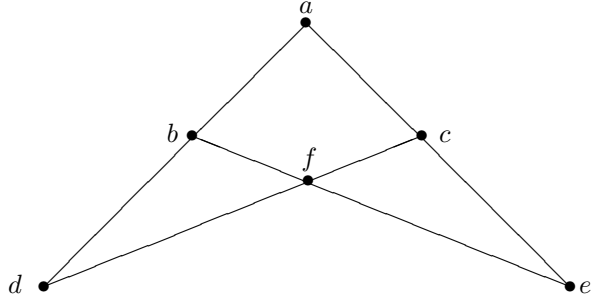
\begin{center}
\begin{picture}(200,110)
\put(0,0){\line(1,1){100}}
\put(200,0){\line(-1,1){100}}
\put(0,0){\line(5,2){143}}
\put(200,0){\line(-5,2){143}}
\put(-13,-2){$d$}
\put(203,-2){$e$}
\put(97,104){$a$}
\put(47,55){$b$}
\put(150,55){$c$}
\put(98,46){$f$}
\put(-2,-2){$\bullet$}
\put(197,-2){$\bullet$}
\put(97,98){$\bullet$}
\put(98,38){$\bullet$}
\put(54,55){$\bullet$}
\put(141,55){$\bullet$}
\end{picture}
\end{center}

\noindent chaque triplet de points non colinéaires forme un ensemble
indépendant et tout point est indépendant de toute droite ne le contenant
pas.

\end{lemma}
Remarquons que $a$ n'est pas forcément algébrique sur $b,c$ (et de même pour
tout autre triplet de points colinéaires).
\begin{proof}

Soient $\alpha$, $\beta$ et $\gamma$ trois génériques
indépendants de $G$ tels que $\alpha$ est algébrique sur un uple $a_0$. En 
prenant successivement des réalisations d'extensions non déviantes,  il
existe des uples $a$, $b$ et $c$
tels que 
$$(a,\alpha) \equiv  (b, \beta) \equiv (c,\gamma) \equiv (a_0,\alpha)$$
et 
$$a \ind_{\alpha} \beta,\gamma \;\;,\;\; b \ind_{\beta}
a,\gamma \; \text{ et } \; c \ind_{\gamma} a,b.$$
 
Vérifions que les uples $a,b,c$ sont indépendants~:
comme $\alpha$, $\beta$ et $\gamma$ sont trois génériques
indépendants, on obtient par transitivité
$$a \ind \beta,\gamma \;\;,\;\; b,\gamma \ind a \; \text{ et } \;
 c
\ind a,b.$$
En particulier, on conclut par transitivité que $a
\ind b, c$ et $b \ind a, c$.

 Prenons maintenant $$d \ind_{\alpha\cdot\beta} a,b,c \;\;,\;\;e
 \ind_{\gamma\cdot\alpha} a,b,c,d \; \text{ et } \; f
 \ind_{\gamma\cdot\alpha\cdot\beta} a,b,c,d,e  $$ tels que
$$(d,\alpha \cdot \beta)
\equiv (e, \gamma \cdot \alpha) \equiv (f,\gamma 
\cdot \alpha\cdot \beta)\equiv ( a_0,\alpha).$$

Montrons d'abord que chaque uple est indépendant des autres sur
l'imaginaire générique correspondant, ce qui permettra de montrer que  dans 
le diagramme précédent, tout point n'appartenant pas à une droite est 
indépendant de celle-ci.

Comme $\gamma\cdot\alpha\cdot\beta$ est algébrique sur $a,b,c$, on a
 $f\ind_{a,b,c,d} e$. Par transitivité, on obtient

$$ e\ind_{\gamma\cdot\alpha} a,b,c,d,f.$$

Puisque $\gamma\cdot\alpha$ est algébrique sur $a,c$, on a
$d\ind_{a,b,c} e$, ce qui donne $d\ind_{\alpha\cdot \beta}
a,b,c,e$. 

Par l'algébricité de $\gamma\cdot\alpha\cdot\beta$   sur $a,b,c$, on a
également  $f\ind_{a,b,c,e} d$ et donc, par transitivité, $$ d\ind_{\alpha\cdot
\beta} a,b,c,e,f.$$

Or, puisque $\alpha\cdot \beta$ est algébrique sur $a,b$,
l'indépendance 
$d\ind_{\alpha\cdot\beta} a,b,c$ donne $c\ind_{a,b} d$ et donc 
$c\ind_{\gamma} a,b,d$. 

Puisque  $\gamma\cdot \alpha$ et $\gamma$
sont interalgébriques sur $a$, de $$
e\ind_{\gamma\cdot\alpha} 
a,b,c,d$$ on obtient par transitivité 
$$c\ind_{\gamma} a,b,d,e \;,$$

puis, comme $\gamma = \gamma \cdot \alpha \cdot \alpha^{-1}$ est algébrique sur
$a,e$, on a  $f\ind_{a,b,d,e} c$, ce qui donne
$$c\ind_{\gamma} a,b,d,e,f.$$

Comme on a vérifié auparavant que les uples $a$, $b$ et $c$ sont 
indépendants,
il suit que $a$ et $b$ ont un comportement symétrique à celui de $c$ et donc
que $$a \ind_{\alpha} b,c,d,e,f \;\text{ et } b \ind_{\beta} a,c,d,e,f.$$

Notons ainsi que tout triplet d'uples non colinéaires
dans le diagramme satisfait les mêmes hypothèses d'indépendances que
le triplet $a,b,c$ au-dessus des génériques correspondants de $G$, qui
sont indépendants. En particulier, tout triplet d'uples non
colinéaires est libre. 

Enfin, vérifions que tout point n'appartenant pas à une droite est 
indépendant de celle-ci. Par symétrie, il suffit de le montrer pour
l'uple $a$ et la droite $b,e,f$~: on a $f\ind_{\gamma\cdot\alpha\cdot
\beta} a,b,e$, d'où
$f\ind_{b,e} a$. Comme $a \ind
b,e$, on conclut que 
$$a  \ind b,e,f.$$

\end{proof}

Pour un uple réel $a$ et un ensemble (éventuellement imaginaire) $A$,
un type $\tp_P(a/A)$ est presque $P$-interne s'il existe $B
\indip_A a$ tel que $a \in \acl_P(A,B,E)$. Notons que $\tp_P(a/A)$ est
presque $P$-interne si et seulement si $\tp_P(\acl_P(A,a)/A)$ l'est. 
De plus, toute extension d'un type presque $P$-interne l'est aussi, et
toute restriction non-déviante d'un type presque $P$-interne le
reste. 

\begin{lemma}\label{L:Pinterne}
Si $\tp_P(a/A)$ est presque $P$-interne sur l'ensemble réel $A$, alors
$a \in \acl(A,E)$.
\end{lemma}
\begin{proof}
On
considère $B=\acl_P(B) \supset A=\acl_P(A)$ tel que $a\indip_A B$  et
$a \in \acl_P(B,E)=\acl(B,E)$. La caractérisation de
l'indépendance donne $a\ind_{A,E} B$ et donc $a \in
\acl(A,E)$.
\end{proof}

\begin{remark}\label{R:genEinterne} 
Si un générique d'un groupe $G$ type-définissable au sens de la paire
est presque $P$-interne, alors $G$ est isogène aux points
$E$-rationnels d'un groupe $T$-type-définissable défini sur $E$. 
\end{remark}
\begin{proof}
Si l'on travaille sur un sous-modèle $M_0$ contenant des représentants
pour les translatés de la composante connexe de $G$, on en déduit que
tous les génériques sont presque $P$-internes.
Soit $a$ une réalisation du générique
principal de $G$ sur $M_0$. Alors $a$ interalgébrique avec $E_{\a}$
au-dessus de $M_0$.
On obtient ainsi dans la proposition
\ref{P:Gdef}, une isogénie entre $G$ et le groupe $H(E)$
$*$-interprétable. On conclut par la remarque
\ref{R:Stab_iso_etoile_def}.
\end{proof}

\begin{definition}\label{D:modere}
Le groupe $T_P$-type-interprétable $G$  est \emph{modéré} si l'un de
ses génériques $\alpha$ est $T_P$-algébrique sur un uple réel $a$,
avec $\tp_P(a/\alpha)$ presque $P$-interne et $a\indip_{\alpha} E$.
\end{definition}
Si un groupe est modéré, tous ses génériques (à l'aide de
paramètres indépendants) satisfont la propriété ci-dessus.

Rappelons que Pillay \cite{Pi07} élimine géométriquement les
imaginaires des belles paires de corps algébriquement clos modulo une
famille explicite de sortes imaginaires. Dans son
introduction, il précise  que ses résultats peuvent s'étendre à toute 
belle paire de structures fortement minimales avec 
$\acl(\emptyset)$ infini. Les résultats préliminaires, plus 
précisément \cite[Lemmata 2.2 \& 2.4]{Pi07}, qui entraînent que tout  
groupe  type-interprétable dans une belle paire est modéré, utilisent 
uniquement l'élimination des imaginaires pour $T$ ainsi que la 
définissabilité du rang de Morley  (fini) pour $T$, ce qui est en 
particulier le cas pour une 
théorie fortement minimale avec $\acl(\emptyset)$ infini, comme par 
exemple, la
théorie des corps algébriquement clos de caractéristique fixée.

Pillay montre que, pour toute
belle paire $(M,E)$ d'une théorie $\omega$-stable $T$ qui a
élimination des imaginaires, tout imaginaire $\alpha$ est
interalgébrique avec
$\alpha'$, qui est lui définissable sur un uple réel $a$ dont le type
$\tp_P(a/\alpha')$ est (presque) $P$-interne, ce qui se généralise au 
cas où $T$ est stable. De plus, si le rang de Morley de $T$ est 
définissable, sa démonstration du Lemma 2.5 assure que $a$ peut être 
choisi indépendant de $E$ sur $\alpha'$. 

Nous ignorons si ce dernier fait 
peut se généraliser au cas où $T$ est stable, soit pour tous les 
imaginaires de la paire, soit uniquement pour les génériques des groupes 
interprétables.

Comme la trace de tout ensemble définissable sur un sous-modèle d'une 
théorie stable l'est à paramètres sur le sous-modèle, on obtient le 
résultat suivant. 

\begin{theorem}\label{T:ppal}
Soit $T_P$ la théorie des belles paires d'une théorie $T$ stable
avec élimination des imaginaires sans la propriété de recouvrement
fini. Tout groupe modéré  $T_P$-type-interprétable $G$ est, à isogénie
près, l'extension des points $E$-rationnels d'un groupe
$T$-type-définissable à paramètres dans $E$  par un groupe 
$T_P$-type-interprétable $N$ quotient d'un groupe 
$T$-type-définissable $V$ par un sous-groupe distingué constitué des 
points $E$-rationnels d'un groupe $T$-type-définissable $N'$ à 
paramètres dans $E$~:

$$\xymatrix @R=.3pc{ 0 \ar[r] &  N \ar[r] &  G \ar[r] &  H(E) \ar[r]
& 0 \\
&& \text{ avec }&&\\
 0 \ar[r] & N'(E) \ar[r]  & V \ar[r] &\ar[r] N  & 0 \,.}$$

\end{theorem}
\begin{proof}
Puisque la presque $P$-internalité est conservée par
extensions non-déviantes, si un groupe $T_P$-type-interprétable 
sur un ensemble de paramètres est modéré, il le reste  
au-dessus d'un sous-modèle contenant ces paramètres. De même pour sa 
composante connexe.  On travaillera donc comme auparavant au-dessus d'un 
sous-modèle. Par la remarque \ref{R:connexe}, on supposera $G$  
connexe  type-interprétable dans une belle paire $(M,E)$ sur un
sous-modèle.

Soient $\alpha$, $\beta$ et $\gamma$ trois génériques
$T_p$-indépendants de $G$. Par hypothèse, le générique $\alpha$ de $G$
est algébrique sur un uple réel $a_0$, qui est $T_P$-indépendant de $E$
au-dessus de $\alpha$. De plus, le type $\tp_P(a_0/\alpha)$ est presque
$P$-interne. Par le lemme \ref{L:gpconfig}, il existe des uples $a$, $b$,
$c$, $d$, $e$ et $f$ avec $$ (\alpha,a) \equiv (\beta, b) \equiv
(\gamma,c)  \equiv
(\alpha\cdot\beta,d) \equiv (\gamma\cdot\alpha,e) \equiv
(\gamma\cdot\alpha\cdot\beta,f)\equiv (\alpha,a_0),$$

tels que, si l'on pose $\bar a = \acl_P(a)$ et de même pour les uples
réels $b$, $d$, $c$, $e$, $f$ dans le diagramme suivant~:

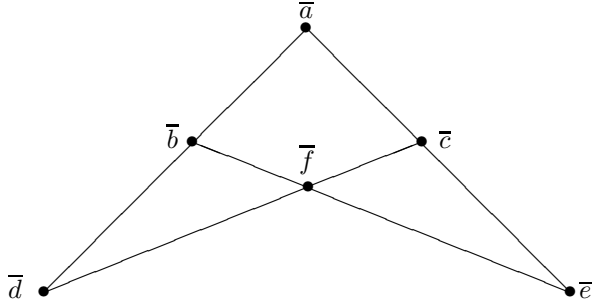
\begin{center}
\begin{picture}(200,110)
\put(0,0){\line(1,1){100}}
\put(200,0){\line(-1,1){100}}
\put(0,0){\line(5,2){143}}
\put(200,0){\line(-5,2){143}}
\put(-13,-2){$\d$}
\put(203,-2){$\e$}
\put(97,104){$\a$}
\put(47,55){$\b$}
\put(150,55){$\cc$}
\put(97,47){$\f$}
\put(-2,-2){$\bullet$}
\put(197,-2){$\bullet$}
\put(97,98){$\bullet$}
\put(98,38){$\bullet$}
\put(54,55){$\bullet$}
\put(141,55){$\bullet$}
\end{picture}
\end{center}

\noindent alors tout triplet de points non colinéaires est $T_P$-indépendant et
tout
point n'appartenant pas à une droite est $T_P$-indépendant de celle-ci.

Comme $a
\indip_\alpha E$, on a $E_{\a} \subset \acl_P(\alpha)\subset \a$,
donc $E_{\a} = \acl_P(\alpha) \cap E$.  De plus, par le lemme
\ref{L:Eamalgam}, comme $\a \indip \b$, on a $E_{\acl_P(\a,\b)}=
\acl(E_{\a},E_{\b})$. En particulier, l'ensemble  $E_{\d}
\subset \acl(E_{\a},E_{\b})$. Ceci est aussi valable pour tout autre
couple de points dans une même droite du diagramme précédent. De
plus, le point $E_{\a}$ est indépendant de la droite
$E_{\acl(\b,\e)}$ et de même pour tout triplet de points
non colinéaires.

La $T$-configuration de groupe :

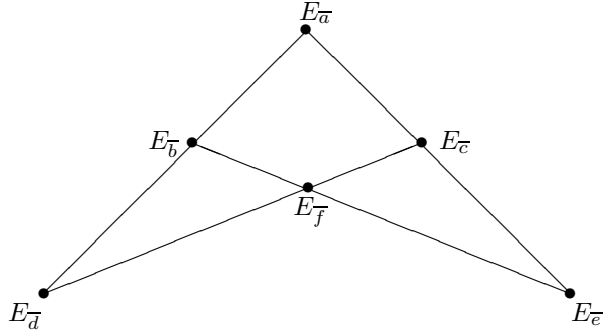
\begin{center}
\begin{picture}(200,110)
\put(0,0){\line(1,1){100}}
\put(200,0){\line(-1,1){100}}
\put(0,0){\line(5,2){143}}
\put(200,0){\line(-5,2){143}}
\put(-13,-10){$E_{\d}$}
\put(200,-10){$E_{\e}$}
\put(97,104){$E_{\a}$}
\put(40,55){$E_{\b}$}
\put(150,55){$E_{\cc}$}
\put(95,30){$E_{\f}$}
\put(-2,-2){$\bullet$}
\put(197,-2){$\bullet$}
\put(97,98){$\bullet$}
\put(98,38){$\bullet$}
\put(54,55){$\bullet$}
\put(141,55){$\bullet$}
\end{picture}
\end{center}

\vskip10pt
\noindent donne ainsi un groupe connexe $*$-interprétable $H$ au sens
du réduit $T$ sur $E$, dont un générique $h$ de $H(E)$ est
$T$-interalgébrique avec $E_{\a}$. Comme dans le cas précédent, on
peut supposer que $H$ est type-définissable et,
à isogénie près, on obtient une
surjection type-définissable $$ \pi: G \to 
H(E).$$ 

Vérifions aussi que les points $\a,\b,\d, \cc, \e, \f$ obtenus
auparavant donnent une $T$-configuration de groupe sur l'ensemble de
paramètres $E$. Tout triplet de points
non colinéaires forme un ensemble indépendant. Puisque $\beta$ est
algébrique sur $\b$ et $ \alpha \cdot \beta$ l'est
sur $\d$, on a $\alpha$ algébrique sur  $b,d$. De plus,
par \ref{L:Eamalgam}, la clôture $T$-algébrique $\acl(\b,\d)$ est
$T_P$-algébriquement close. Le type $\tp_P(\a/\acl(\b,\d))$ est aussi
presque $P$-interne, donc  l'uple $\a$ est $T$-algébrique sur
$E,\b,\d$, par le lemme \ref{L:Pinterne}.
De même pour les autres droites.

On obtient un groupe connexe $*$-interprétable $W$ au sens du réduit
$T$ sur $E$ et deux génériques indépendants $\a'$ et $\b'$ de $W$ tels
que
$\a'$ est $T$-interalgébrique avec $\a$ sur $E$, l'élément $\b'$ est
$T$-interalgébrique avec $\b$ sur $E$ et $\a' \cdot \b'$ l'est
avec
$\d$. Comme les uples $\alpha$, $\beta$ et $\alpha \cdot \beta$
sont
déjà algébriques sur des sous-uples finis de $\a'$, $\b'$ et 
$\a'\cdot\b'$, respectivement, et puisque $W$ est une limite
projective de groupes connexes type-définissables au sens du réduit
$T$ sur $E$, alors il existe un groupe connexe $V$
type-définissable au sens du réduit $T$ sur $E$ contenant deux
génériques indépendants $a_1$ et $b_1$ tels que $\alpha$ (resp.
$\beta$, $ \alpha \cdot \beta$) est algébrique sur $a_1$ (resp.
$b_1$, $a_1\cdot b_1$). Notons que $a_1$ est $T$-algébrique sur
$E,\a$, et de même pour les autres points.

Comme $\a$, $\b$ et $\d$ sont deux à deux $T_P$-indépendants
sur $E$, les uples $a_1$, $b_1$ et $a_1\cdot b_1$ le sont aussi.

Puisque $\alpha$ est algébrique sur $a$ et $a \indip_{E_{\a}} E$, on
conclut que $\alpha\indip_{E_{\a}} E$. Si $N=\ker(\pi)^0$, l'élément
$\alpha$ est générique dans $N\cdot \alpha$ sur $\acl(h)=E_{\a}$ et
donc générique dans $N\cdot \alpha$ sur $E$. 

Le lemme \ref{L:Stab_iso} appliqué aux paires $(a_1,\alpha)$ et
$(b_1,\beta)$ nous donne ainsi une surjection type-définissable $\phi$ de 
$V$ sur $N$ (à isogénie près). Il nous reste à montrer que la composante 
connexe $N_1$  de $\ker(\phi)$ est isogène aux points
$E$-rationnels d'un groupe $T$-type-définissable $N'$. Ceci suit de la 
remarque \ref{R:genEinterne}~; en effet, considérons $n_1$ un générique
de $N_1$ sur $E,a_1$. Alors, le point $(n_1,1_N)$ est dans le
stabilisateur de $\tp_P(a_1,\alpha/\aclq_P(E))$, donc $n_1 \cdot
a_1\equiv_{\alpha}^P a_1$. Puisque
$\tp_P(a_1/\alpha)$ est
presque $P$-interne, car $a_1$ est algébrique sur $E,a$, alors le type
$\tp_P(a_1/\alpha)=\tp_P(n_1\cdot a_1/\alpha)$ l'est aussi. Comme
$\alpha$ est algébrique sur $E,a_1$,  le type
$\tp_P(n_1/E,a_1)$ est également presque $P$-interne ainsi que 
$\tp_P(n_1/E)$, grâce à l'indépendance $n_1\indip_E a_1$.
\end{proof}

Comme la théorie de belles paires $(K,F)$ de corps algébriquement clos
est $\omega$-stable, tout groupe type-interprétable est interprétable.
De
plus,  Delon  montre qu'elle est
modèle-complète dans une expansion naturelle du langage des anneaux
\cite[Corollaire 15]{fD12}. En
particulier, si $ A$ est une partie de $K$ non incluse dans $F$ et 
comme la clôture algébrique $\acl_P(A)$ est $P$-indépendante, cette partie $A$
est un
sous-modèle de $(K,F)$ au sens de l'expansion et donc une sous-structure
élémentaire au sens du langage de la paire, ce qui permet de
montrer le résultat suivant~:

\begin{cor}\label{C:intACF}
Tout groupe $G$ interprétable dans une paire $(K,F)$ de corps
algébriquement clos où $K$ est une extension propre de $F$ est, à
isogénie près, l'extension
des points $F$-rationnels d'un groupe algébrique défini sur $F$ par
un groupe $N$ quotient d'un groupe algébrique $V$ par un
sous-groupe distingué $N'(F)$, constitué des points rationnels d'un
groupe algébrique défini sur $F$
$$\xymatrix @R=.3pc{ 0 \ar[r] &  N(K)
\ar[r] &  G(K) \ar[r] &  H(F) \ar[r]
& 0 \\
&& \text{ avec }&&\\
 0 \ar[r] & N'(F) \ar[r]  & V(K) \ar[r] &\ar[r] N(K)  & 0 \,,}$$
où
$H$ et $N'$ sont des groupes algébriques définis sur $F$. 
\end{cor}
Si le groupe $G$ est interprétable sur un sous-corps $k\not\subseteq
F$, alors les groupes $V$ et $N$ obtenus le sont sur $\acl_P(k)\subset
\overline{kF}$.

\end{document}